%% file: sppShift_arxiv2.tex
\documentclass[a4paper,12pt]{article}
\usepackage[utf8]{inputenc}
\usepackage{amsmath}
\usepackage{amssymb}
\usepackage{amsthm}
\usepackage{cite}
\usepackage{graphicx}
\usepackage{tikz}
\usepackage{pgfplots}

\pgfplotsset{compat=1.17}

\makeatletter
\renewcommand*\env@matrix[1][r]{\hskip -\arraycolsep
  \let\@ifnextchar\new@ifnextchar
  \array{*\c@MaxMatrixCols #1}}
\makeatother

\newcommand{\e}{\mathrm{e}}
\newcommand{\pt}{\partial}
\newcommand{\eps}{\varepsilon}
\newcommand{\norm}[2]{\|{#1}\|_{#2}}
\newcommand{\tnorm}[1]{\left|\!\!\;\left|\!\!\;\left| {#1}
                       \right|\!\!\;\right|\!\!\;\right|}

\newcommand{\U}{\mathcal{U}}
\newcommand{\PS}{\mathcal{P}}
\newcommand{\scp}[1]{\left\langle #1 \right\rangle}

\newcommand{\dx}{\mathrm{d}x}
\newcommand{\jump}[1]{[\hspace*{-2pt}[#1]\hspace*{-2pt}]}
\newcommand{\Qm}[1]{Q_m\left[#1\right]}
\newcommand{\tmi}{t_{m,i}}

\theoremstyle{plain}
\newtheorem{theorem}{Theorem}[section]
\newtheorem{lemma}[theorem]{Lemma}

\newtheorem{remark}[theorem]{Remark}
\newtheorem{Notation}{Notation}
\title{A balanced norm error estimation for the time-dependent reaction-diffusion problem with shift in space
	}
\author{Mirjana Brdar,\footnote{Faculty of Technology Novi Sad, University of Novi Sad, Serbia,\newline \mbox{e-mail}: mirjana.brdar@uns.ac.rs}
        \,
        Sebastian Franz,\footnote{Institute of Scientific Computing, Technische Universit\"at Dresden, Germany,\newline  \mbox{e-mail}: sebastian.franz@tu-dresden.de}
        \,
        Lars Ludwig,\footnote{Institute of Scientific Computing, Technische Universit\"at Dresden, Germany,\newline  \mbox{e-mail}: lars.ludwig@tu-dresden.de}
        \,
        Hans-G\"{o}rg Roos,\footnote{Institute of Numerical Mathematics, Technische Universit\"at Dresden, Germany,\newline  \mbox{e-mail}: hans-goerg.roos@tu-dresden.de}}

\begin{document}
\maketitle

\begin{abstract}
 We consider a singularly perturbed time-dependent problem with a shift term in space. On appropriately
    defined layer adapted meshes of Dur\'{a}n- and S-type we derive a-priori error estimates for the stationary
    problem. Using a discontinuous Galerkin method in time we obtain error estimates for the full discretisation.
    Introduction of a weighted scalar products and norms allows us to estimate the error of the time-dependent problem in energy
    and balanced norm. So far it was open to prove such a result. Error estimates in the energy
norm is for the standard finite element discretization in space, and for
the error estimate in the balanced norm the computation of the numerical
solution is changed by using a different bilinear form. Some numerical results are given to confirm the
    predicted theory and to show the effect of shifts on the solution.
\end{abstract}



%----------------------------------------------------------------------
\section{Introduction}
%----------------------------------------------------------------------

\label{sec:intro}
  Singularly perturbed problems with some kind of shifts often represent mathematical models of various phenomena
  in the biosciences and control theory, \cite{Glizer, LongtinMilton, Wazewska}.
  In this paper we are interested in time-dependent singularly perturbed reaction-diffusion problems with large shifts
  in space that arise especially in the theoretical analysis of neuronal variability, \cite{Fienberg, Segundo} for example in determination of
  the behaviour of a neuron to random synaptic inputs. The first paper on this subject was given by Stein in 1965 \cite{Stein},
  where he proposed a practical model for the stochastic effects due to the neuronal variability and obtained the approximate
  solution to the differential-difference equation model using Monte Carlo techniques. After more than two decades of Stein's
  work, Musila and L\'{a}nsk\'{y} generalised his model \cite{MusilaL}. Lange and Miura were the first who consider various
  classes of singularly perturbed ordinary differential equations where small shifts appear in the solution and its first
  derivative and they used an asymptotic approach for analysis, \cite{LM1, LM2, LM3}. So far, mostly this type of problem
  (stationary problem) has been solved by the method of central differences, of which in some papers Taylor's series was
  used \cite{KumarKadalbajoo, Gupta}, and in some of them the solution was obtained directly \cite{Bansal, Chakravarthy, Devendra}.

  There are only a few papers that use the finite element method. For the stationary problem,
  Nicaise and Xenophontos in \cite{NX13} use the hp-version of the standard finite element method on a layer-adapted mesh and Zarin in \cite{Zarin14} uses the
  more complicated non-symmetric discontinuous Galerkin (dG) method with interior penalties.

  In this paper we consider the following parabolic singularly perturbed problem with a shift, which represents a
  generalization of the parabolic equation from \cite{Devendra}: Find $u$ such that
  \begin{subequations}
     \begin{align}
      Lu(x,t)&\equiv \pt_t u(x,t) - \varepsilon^2  \pt_x^2u(x,t) + a(x,t) u(x,t)+b(x,t)u(x-1,t)\notag\\&
          =f(x,t), \quad (x,t)\in D, \label{1}\\[1ex]
      u(x,0)&=u_0(x), ~\quad x\in\bar{\Omega},\label{2}\\[1ex]
      u(x,t)&=\Phi(x,t), \quad (x,t)\in D_L=\{(x,t):-1< x\leq 0; t\in\bar{\Lambda}\}\label{3}\\[1ex]
      u(x,t)&=\Psi(x,t), \quad (x,t)\in D_R=\{(2,t): t\in\bar{\Lambda}\}\label{4}
    \end{align}
  \end{subequations}
  where $\varepsilon\in(0,1]$ is a small perturbation parameter and $D=\Omega \times\Lambda=(0,2)\times(0,T]$.
  The functions $a,$ $b,$ $f,$ $\Phi,$ $\Psi$ and $u_0$ are sufficiently smooth, bounded and independent of $\varepsilon.$ It is also assumed that $a$ and $b$ satisfy
  \begin{align}\label{5}
    a(x,t)\geq \alpha^2>0%, \quad  b_0<b(x)<b_1
    \quad \mbox{and} \quad \alpha^2-\|b(x,t)\|_{\infty}\geq\gamma>0, \quad x\in\bar{\Omega},\, t\in (0,T]
  \end{align}
  where $\alpha$ and $\gamma$ are constants.
  The solution of this problem is characterised by two exponential boundary layers near $x=0$ and $x=2$,
  and a third, inner layer near $x=1$ may appear depending on the function $b$. 
\begin{figure}[ht]
    \begin{center}
      \includegraphics[scale=0.5]{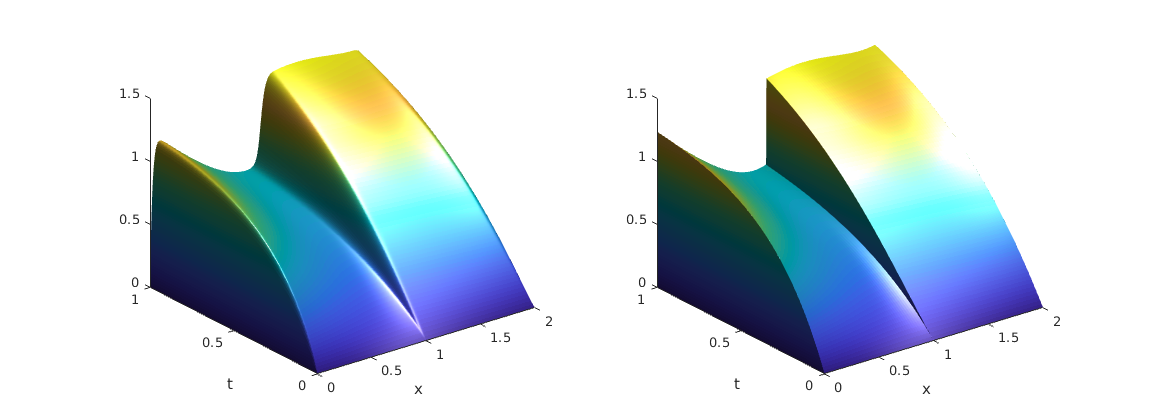}
         \end{center}
    \caption{Graph of the solution %Approximate solution %in $V_{1/80, 3}^{1/192, 2}$
            for $\varepsilon=0.04$ (left) and $\varepsilon=0.001$ (right) \label{fig:1}}
  \end{figure}  
%
%  \begin{figure}[ht]
%    \begin{center}
%       \includegraphics[scale=0.4]{pics/solPhi0_04_.png}
%       \includegraphics[scale=0.4]{pics/solPhi1em3_.png}\\
 %     \includegraphics[scale=0.4]{pics/solPhi0_04.png}
 %     \includegraphics[scale=0.4]{pics/solPhi1em3.png}
 %   \end{center}
%    \caption{Graph of the solution %Approximate %solution %in $V_{1/80, 3}^{1/192, 2}$
 %           for $\varepsilon=0.04$ (left) and $%%\varepsilon=0.001$ (right) \label{fig:1}}
  %\end{figure}
%
  To give a quick idea how a solution to (1) may look like, Figure \ref{fig:1} depicts two
space-time solutions of the numerical example in Section 4 for two different values
of $\eps$.

  We use for the discretisation in space the standard finite element method on two different classes of
  layer-adapted meshes. The introduction of a weight $\beta$, similarly to \cite{MS21}, allows to
  analyse the discretisation error in space not only in the energy norm, but, additionally, in a balanced
  norm, which reflects the layer behavior of the solution better than the energy norm.
  For the full discretisation we prefer the time discretisation with the discontinuous Galerkin (dG) method due to its
  flexibility and the possibility of arbitrarily high order.
  In the context of parabolic equations this method was already introduced and analysed 1978, see also \cite{Jamet, EJT, Thomee}.

  Recently in \cite{FrM16} the authors gave
  a general numerical analysis of time-dependent singularly perturbed problems discretised
  in time by the dG method and applicable to numerous spatial operators. We will apply the results of
  that paper in the analysis of our problem.
  The introduction of the weight $\beta$ yields a balanced error estimation for the time-dependent
  problem. So far it was open to prove such a result, even for the time-dependent reaction-diffusion case. This new approach is not
  restricted to the special shift problems studied here, but can also be used for other time-dependent singularly perturbed problems.

  The paper is organised as follows: In Section \ref{sec:stat} we present an analysis for the stationary singularly perturbed shift problem,
  including information on a solution decomposition together with layer adapted meshes, weighted norm, finite element method. % and numerical analysis on S-type and Duran meshes.
  The main results for the time-dependent problems are in the Section \ref{sec:full} and are experimentally verified in Section \ref{sec:numerics}.
  \vspace{2mm}

  \begin{Notation}
    For a set $D$, we use the standard notation
    of Sobolev spaces, where $\|\cdot\|_{0,D}$ is the $L^2-$norm, $\|\cdot\|_{k,D}$ is the full $H^k-$norm and
    $|\cdot|_{k,D}$ is the seminorm.
    The standard scalar product in $L_2(D)$ is marked with $\scp{\cdot,\cdot}_D$.
    If $D=\Omega$ we drop the $\Omega$ from the notation.
%     Throughout the paper, $C$ will denote a generic positive constant independent of the perturbation parameter $\varepsilon$ and the mesh.
    Throughout the paper, we will write $A\lesssim B$ if there exists a generic positive constant $C$ independent of the perturbation parameter $\varepsilon$ and the mesh,
    such that $A\leq C B$. We will also write $A\sim B$ if $A\lesssim B$ and $B\lesssim A$.
  \end{Notation}

%%%%%%%%%%%%%%%%%%%%%%%%%%%%%%%%%%%%%%%%%%%%%%%%%%%%%%%%%%%%%%%%%%%%%%%%%%%%%%%%%%%%%%%%%%%%%%%%%%%%%%%%%%%%%%

\section{The stationary problem}\label{sec:stat}
  Let us consider the stationary problem
  \begin{subequations}\label{eq:stat}
    \begin{align}
      -\eps^2\pt_x^2 u(x) + a(x)u(x) + b(x)u(x-1)&=f(x),\quad 0<x<2,\\
      u(x)&=\Phi(x),\quad -1<x\leq0,\\
      u(2)&=0,
    \end{align}
  \end{subequations}
  where $a\geq\alpha^2>0$ and $\alpha^2-\norm{b}{\infty}\geq\gamma>0$. Without loss of generality
  we can assume $\Phi(0)=0$ as otherwise the transformation
  \[
    u(x)=\tilde u(x)+\left( 1-\frac{x}{2} \right)\Phi(0)
  \]
  gives a problem for $\tilde u$ with homogeneous boundary conditions and changed data $f$ and $\Phi$.

  \subsection{Solution decomposition and spatial mesh}
  In \cite{NX13} the problem \eqref{eq:stat} was considered in the case of $b_1<b<b_2<0$
  and a solution decomposition was obtained. By the same analysis the more general case of a bounded $b$
  can be analysed. Here we want to emphasize that for negative $b$ the maximum principle holds, but for positive $b$ we should proceed as in the convection-diffusion case, see \cite{BFLR}.
  Assuming $f$ to be smooth enough, we have
  \begin{subequations}
  \begin{equation}\label{sd}
    u=s+w_1+w_2+w_3,
  \end{equation}
  where for a given $k>0$ and all $\ell\in\{0,\dots,k+1\}$ it holds
  \begin{align}
    \norm{\pt_x^\ell s}{L^\infty(0,2)}&\lesssim 1,\label{s}\\
    |\pt_x^\ell w_1(x)|&\lesssim \eps^{-\ell}\e^{-\frac{\alpha x}{\eps}},&
    |\pt_x^\ell w_2(x)|&\lesssim \eps^{-\ell}\e^{-\frac{\alpha |x-1|}{\eps}},&
    |\pt_x^\ell w_3(x)|&\lesssim \eps^{-\ell}\e^{-\frac{\alpha (2-x)}{\eps}}.\label{exp}
  \end{align}
  \end{subequations}
  Thus, $w_1$ and $w_3$ are the layer functions corresponding to the left and right boundary, resp., while
  $w_2$ is an interior layer function, and $s$ represents the smooth part.

  Having the knowledge about the layer structure, we can construct a layer-adapted mesh $\Omega_h$ that resolves the
  layers. We will consider two types of meshes: S-type meshes (analogously the analysis on B-type meshes works) and graded Dur\'{a}n mesh.

\subsubsection{S-type meshes}

  For the construction of an S-type mesh, see \cite{RL99}, let us assume the number of cells $N$ to be divisible by 8.
  Next we define a mesh transition value
  \[
    \lambda = \frac{\sigma\eps}{\alpha}\ln(N),
  \]
  with a constant $\sigma$ to be specified later. Then using a monotonically increasing mesh defining function $\phi$ with $\phi(0)=0$ and $\phi(1/2)=\ln(N)$,
  see \cite{RL99} for the precise conditions on $\phi$,
  we construct the mesh nodes
%   \begin{equation}\label{Smesh}
% x_i=\left\{\begin{array}{lll}
%           \displaystyle\frac{\sigma\eps}{\alpha}\phi\left(\frac{4i}{N}\right),&   i=0,\dots,N/8\\[0.5ex]
%            \displaystyle\frac{4i}{N}(1-2\lambda)+2\lambda-\frac{1}{2},&  i=N/8,\dots,3N/8\\ [0.5ex]
%           \displaystyle 1-\frac{\sigma\eps}{\alpha}\phi\left(2-\frac{4i}{N}\right), &  i=3N/8,\dots,N/2\\[0.5ex]
%             \displaystyle 1+x_{i-N/2},&  i=N/2,\dots,N.\\[0.5ex]
%          \end{array}
%           \right.
% \end{equation}
  \begin{equation}\label{Smesh}
    x_i=\begin{cases}
           \frac{\sigma\eps}{\alpha}\phi\left(\frac{4i}{N}\right),&       0\leq i \leq \frac{N}{8},\\
           \frac{4i}{N}(1-2\lambda)+2\lambda-\frac{1}{2},& \frac{N}{8}\leq i\leq \frac{3N}{8},\\
           1-\frac{\sigma\eps}{\alpha}\phi\left(2-\frac{4i}{N}\right),&   \frac{3N}{8}\leq i \leq \frac{N}{2},\\
           1+x_{i-N/2},& \frac{N}{2}\leq i\leq N.
        \end{cases}
  \end{equation}
  Let us denote the smallest mesh-width inside the layers by $h$. For this holds $h\leq \eps$.
  Note that associated with $\phi$ is the mesh characterising function $\psi=\e^\phi$, that classifies the convergence quality
  of the meshes by the quantity $\max|\psi'|:=\max\limits_{t\in[0,1/2]}|\psi'(t)|$. Two of the most common S-type meshes are
  the Shishkin mesh with
  \[
    \phi(t) = 2t\ln N,\quad
    \psi(t) = N^{-2t},\quad
    \max|\psi'| =2\ln N
  \]
  and the Bakhvalov-S-mesh
  \[
    \phi(t)=-\ln(1-2t(1-N^{-1})),\quad
    \psi(t)=1-2t(1-N^{-1}),\quad
    \max|\psi'|= 2.
  \]

  \subsubsection{Dur\'{a}n meshes}

  We will also use the recursively defined graded mesh that Dur\'{a}n and Lombardi introduced in \cite{Duran}, modified to our problem.
  Its advantages are the simple construction and generation of mesh points (without transition point(s))
  and some robustness property. Precisely, when we approximate a singularly perturbed problem with an a priori adapted mesh,
  it is natural to expect that a mesh designed for some value of the small parameter will also work well for larger values of it.
  In this regard the recursively graded meshes have better behavior in numerical experiments.

  In the construction of this mesh, we first define the points on the interval $[0,1]$, and than on the rest of the domain.
  Let $H\in(0,1)$ be arbitrary and define the number M from the conditions
  \begin{subequations}\label{M}
    \begin{equation}
      H\eps(1+H)^{M-2}<\frac{1}{2}\leq%\,\,\, \mbox{and}\,\,\,
      H\eps(1+H)^{M-1}%\geq\frac{1}{2}.
    \end{equation}
    which are equivalent to
    \begin{equation}
      M=\left\lceil1-\frac{\ln(2H\eps)}{\ln(1+H)}\right\rceil.
    \end{equation}
  \end{subequations}
  On the interval $[0,1]$ we define mesh points recursively in the following way:
  \begin{equation}\label{Dmesh1}
    x_i=\begin{cases}
          0, & i=0,\\
          H\eps,& i=1,\\
          (1+H)x_{i-1}, & 2\leq i\leq M-1,\\
          1/2,& i=M,\\
          1-x_{2M-i},& M+1\leq i\leq 2M-1,\\
          1, &i=2M.
        \end{cases}
  \end{equation}
  In case the interval $(x_{M-1}, 1/2)$ is too small compared to $(x_{M-2}, x_{M-1})$, we simply omit the mesh point $x_{M-1}$.
  In the rest of the domain, i.e. on the interval $[1, 2]$, the Dur\'{a}n mesh is given by:
  \begin{equation}\label{Dmesh2}
    x_{2M+i}=1+x_i,\,\,\,0\leq i\leq 2M.
  \end{equation}
  The total number of mesh subintervals is $N=4M$
  and depends on the parameters $H$ and conditions \eqref{M}, see \cite{Duran}. Moreover, the inequality
  \begin{equation}\label{H}
    H\lesssim N^{-1}\ln(1/\eps)
  \end{equation}
  applies.

  The mesh step sizes $h_i=x_i-x_{i-1}$, $1\leq i\leq 4M$, can be estimated with $CH\eps\leq h_i\leq H,$ and satisfy
  \begin{equation}\label{eq:step}
    \begin{array}{lll}
      h_i=H\eps, & i\in\{1,2M,2M+1,4M\},\\[0.5ex]
      h_i\leq Hx,& 2\leq i\leq M,\\ [0.5ex]
      h_i\leq H(1-x), & M+1 \leq i\leq 2M-1,\\[0.5ex]
      h_i\leq H(x-1), & 2M+2\leq i\leq 3M,\\ [0.5ex]
      h_i\leq H(2-x), & 3M+1\leq i\leq 4M-1,
    \end{array}
%     \begin{array}{lll}
%       h_i=H\eps, & i\in\{1,2M,2M+1,4M\},\\[0.5ex]
%       h_i\leq Hx,& 2\leq i\leq M\,\,\text{and}\,\,2M+1\leq i\leq 3M,\\ [0.5ex]
%       h_i\leq H(1-x), & M+1 \leq i\leq 2M-1\,\,\text{and}\,\,3M+1\leq i\leq 4M.
%     \end{array}
  \end{equation}

  \subsection{Weak formulation, norm and method}
  The analysis of finite element methods for singularly perturbed problems is usually conducted in the so called energy norm, the norm naturally associated
  with the bilinear form. In \cite{MS21} a weighted scalar product was used for a reaction-diffusion problem instead,
  and the corresponding weighted norm is stronger and captures the layer behaviour of the solution better, see Rem.~\ref{rem:balanced}.
  We will perform both analyses for the shift problem at the same time by considering a weight $\beta$,
  that fulfills the following conditions:
  \begin{subequations}\label{eq:beta}
    \begin{align}
      \beta(x)&\geq 1,\\
      \exists \pt_x\beta(x)\text{ with }|\pt_x\beta(x)|&\leq \frac{\alpha}{\eps}\beta(x)\text{ a.e. in }\Omega,\\
      \int \beta(x)\dx &\lesssim 1.
    \end{align}
  \end{subequations}
  The two examples for such weights we are interested in are
  \begin{itemize}
    \item $\beta_e(x)=1$, which gives raise to the classical analysis in unweighted $L^2$-spaces and the energy norm,
    %\item $\displaystyle \beta_b(x)=1+\frac{1}{\eps}\left( \exp(-\alpha x/\eps)+\exp(-\alpha|x-1|/\eps)+\exp(-\alpha(2-x)/\eps) \right)$,
    \item $\beta_b(x)=1+\frac{1}{\eps}\left( \e^{-\alpha x/\eps}+\e^{-\alpha|x-1|/\eps}+\e^{-\alpha(2-x)/\eps} \right)$,
          which uses the structure of the layers and is similar to the weight given in \cite{MS21} for reaction diffusion problems.
  \end{itemize}

  Assuming $u,\,v\in H^1_0(\Omega)$, we obtain after multiplying \eqref{eq:stat} by $\beta$ the weighted bilinear form
  \begin{align}
    B_\beta(u,v)
      &:=\eps^2\scp{\pt_x u,\pt_x (\beta v)}_{\Omega}+\scp{au,v}_{\beta,\Omega}+\scp{bu(\cdot -1),v}_{\beta,(1,2)}\notag\\
      &= \scp{f,v}_{\beta,\Omega}-\scp{b\Phi(\cdot-1),v}_{\beta,(0,1)}=:F_\beta(v),\label{eq:weak}
  \end{align}
  where
  \[
    \scp{a,b}_{\beta,I}:=\int_I a(x)b(x)\beta(x)\dx.
  \]
  Associated to this weighted scalar product is the weighted $L^2$-norm
  \[
    \norm{a}{\beta}^2:=\scp{a,a}_{\beta,\Omega}
  \]
  and the weighted triple-norm
    \[
      \tnorm{v}^2_\beta:=\eps^2\norm{\pt_x v}{\beta}^2+\gamma\norm{v}{\beta}^2.
    \]
  \begin{lemma}\label{lem:coercivity}
    The bilinear form $B_\beta$ is coercive w.r.t. the weighted triple-norm, i.e.
    it holds for all $v\in H_0^1(\Omega)$
    \begin{align}
      B_\beta(v,v)\geq\frac{1}{2}\tnorm{v}^2_{\beta}.\label{eq:coer}
    \end{align}
  \end{lemma}
  \begin{proof}
    We have
    \[
      B(v,v)
        \geq \eps^2\norm{\pt_x v}{\beta}^2+\eps^2\scp{\pt_x v,(\pt_x \beta) v}_{\Omega}+ \alpha^2\norm{v}{\beta}^2+\scp{bv(\cdot-1),v}_{\beta,(1,2)}.
    \]
    With
    \begin{align*}
      |\eps^2\scp{\pt_x v,(\pt_x\beta) v}_{\Omega}|
        &\leq \eps^2\scp{|\pt_x v|,|\pt_x\beta||v|}_{\Omega}
         \leq \eps\alpha\scp{|\pt_x v|,|v|}_{\beta,\Omega}
         \leq \frac{\eps^2}{2}\norm{\pt_x v}{\beta}^2+\frac{\alpha^2}{2}\norm{v}{\beta}^2,\\
      |\scp{bv(\cdot -1),v}_{\beta,(1,2)}|
        &\leq \frac{\norm{b}{\infty}}{2}\left( \norm{v}{\beta,(0,1)}^2+\norm{v}{\beta,(1,2)}^2 \right)
         = \frac{\norm{b}{\infty}}{2}\norm{v}{\beta}^2
    \end{align*}
    we finish the proof by
    \[
      B(v,v)
        \geq \frac{1}{2}\eps^2\norm{\pt_x v}{\beta}^2 + \frac{1}{2}\alpha^2\norm{v}{\beta}^2 - \frac{\norm{b}{\infty}}{2}\norm{v}{\beta}^2
        \geq \frac{1}{2}\tnorm{v}^2_\beta.
    \]
  \end{proof}
  \begin{remark}\label{rem:balanced}
    Let us consider the norm $\tnorm{\cdot}_\beta$
    for the smooth part $s$
    and
    for a typical layer function $w(x)=\exp(-\alpha x/\eps)$,
    and the two choices of $\beta$ presented above. For $\beta=\beta_e=1$ we have
    \[
      \tnorm{s}_{\beta_e} \sim 1,\quad
      \tnorm{w}_{\beta_e} \sim \eps^{1/2} \stackrel{\eps\to 0}{\longrightarrow}0.
    \]
    Thus, this so called energy norm does not see the layer function for small $\eps$ -- it is unbalanced.
    Over the last years many researchers
    considered convergence in a balanced norm, where the norm of boundary layers does not vanish for $\eps\to 0$.
    Examples are \cite{LS11,RSch15,FrR19,AMM19,MS21}. Following \cite{MS21} the second given weight function $\beta_b$ is constructed.
    Here it holds
    \[
      \tnorm{s}_{\beta_b} \sim 1,\quad
      \tnorm{w}_{\beta_b} \sim 1
    \]
    independent of $\eps$ -- the norm is balanced.
  \end{remark}

  As discrete space we use the piecewise polynomial space
  \[
    \U_h:=\{v\in H_0^1(\Omega)\colon v|_K\in\PS_k(K)\,\forall K\in\Omega_h\},
  \]
  where $\PS_k(K)$ denotes the space of polynomials of degree at most $k$ on the cell $K\in \Omega_h$.
  Our discrete method now reads: Find $u_h\in\U_h$ s.t. for all $v\in\U_h$ it holds
  \[
    B_\beta(u_h,v)=F_\beta(v).
  \]
  As a consequence of this and \eqref{eq:weak} we have Galerkin orthogonality
  \begin{align}
    B_\beta(u-u_h,v)=0\quad\text{for all }v\in\U_h.\label{eq:orth}
  \end{align}

  \subsection{Error estimation}
  In order to analyse the method we split the error into a discrete error and an interpolation error.
  For this purpose, let $I:C(\Omega)\to\U_h$ denote the piecewise Lagrange interpolation into $\U_h$
  using locally equidistantly distributed points (any other reasonable choice of functionals also suffices).
  Then we set
  \[
    u-u_h=\eta-\xi,\,\text{where }\eta:=u-Iu\text{ and }\xi=u_h-Iu\in\U_h
  \]
  and it follows
  \[
    \tnorm{u-u_h}_\beta\leq\tnorm{\eta}_\beta+\tnorm{\xi}_\beta.
  \]
  By \eqref{eq:coer} and \eqref{eq:orth} we obtain
  \begin{align}
    \frac{1}{2}\tnorm{\xi}_\beta^2
      &\leq B_\beta(\xi,\xi)
        = B_\beta(\eta,\xi)\notag\\
      &\leq \eps^2\norm{\pt_x\eta}{\beta}\norm{\pt_x\xi}{\beta}
            +\eps\norm{\pt_x\eta}{\beta}\norm{\xi}{\beta}
            +\norm{a}{\infty}\norm{\eta}{\beta}\norm{\xi}{\beta}
            +\norm{b}{\infty}\norm{\eta}{\beta}\norm{\xi}{\beta}\label{eq:bound}\\
      &\leq \left(1+\frac{2}{\gamma} \right)\eps^2\norm{\pt_x\eta}{\beta}^2
           +\frac{2(\norm{a}{\infty}+\norm{b}{\infty})^2}{\gamma}\norm{\eta}{\beta}^2
           +\frac{1}{4}\tnorm{\xi}_\beta^2. \notag
  \end{align}
  Thus we have
  \[
    \tnorm{\xi}_\beta%\leq \eps^2\norm{\pt_x\eta}{\Omega}^2+\frac{2}{\gamma}\left( \norm{a}{\infty}^2+\norm{b}{\infty}^2 \right)\norm{\eta}{\Omega}^2
    \lesssim \tnorm{\eta}_\beta
  \]
  and therefore
  \begin{equation}\label{eq:conv}
    \tnorm{u-u_h}_\beta\lesssim \tnorm{u-Iu}_\beta.
  \end{equation}
  In other words, the method is quasi-optimal. The discretization error is smaller than the interpolation error, so it is enough to estimate the interpolation error to estimate the error between the correct and the approximate solution.

  \subsubsection{The interpolation error on S-Type meshes}
  On an S-type mesh we have for $\sigma\geq k+1$, see \cite{Zarin14}
  \begin{subequations}\label{eq:inter:S}
    \begin{align}
      \norm{u-Iu}{\beta_e}&\lesssim (h+N^{-1}\max|\psi'|)^{k+1},\\
      \tnorm{u-Iu}_{\beta_e}&\lesssim (\eps^{1/2}+N^{-1})(h+N^{-1}\max|\psi'|)^k.
    \end{align}
  \end{subequations}
  Note that the additional factor of $\eps^{1/2}$ in the triple norm indicates that the norm is not balanced
  and we obtain a higher convergence order than to be expected by a using elements of degree $k$.
  For the second choice of $\beta$ we find in \cite{MS21} an interpolation error result on a
  Shishkin mesh and (bi-)linear elements for a reaction-diffusion problem.
  We extend the analysis here for higher order elements and S-type meshes, following their proof
  as the layer structure is similar.
  \begin{lemma}\label{lem:main:S}
    For the standard piecewise Lagrange interpolation operator $I$ on an S-type mesh with
    $\sigma\geq k+1$ we obtain
    \begin{align}\label{eq:inter:TS2}
%       \tnorm{u-Iu}_{\beta_b}\lesssim (h+N^{-1}\max|\psi'|)^k.
      \tnorm{u-Iu}_{\beta_b}\lesssim \eps h^k+(N^{-1}\max|\psi'|)^k.
    \end{align}
  \end{lemma}
  \begin{proof}
    Proving interpolation error results in the weighted norm can be done by using the
    assumptions on $\beta$ and the $L^\infty$-norm interpolation error bounds. We have
    \begin{align*}
      \norm{u-Iu}{\beta_b}^2
        &=\sum_{K\in\Omega_h}\int_K(u-Iu)^2\beta_b
        \leq\norm{u-Iu}{L^\infty(\Omega)}^2\int\beta_b
        \lesssim \norm{u-Iu}{L^\infty(\Omega)}^2\\
      \norm{\pt_x(u-Iu)}{\beta_b}^2
        &\lesssim \norm{\pt_x(u-Iu)}{L^\infty(\Omega)}^2,
    \end{align*}
    where $K\in\Omega_h$ denotes the intervals in the mesh $\Omega_h$ over $\Omega$.
    Now the interpolation error in the $L^\infty$-norm follows in the classical way of applying
    local interpolation error estimates, %see \cite{Apel99} for the anisotropic version in two and three dimensions,
    the solution decomposition and results for $\sigma\geq k+1$ in
    \begin{subequations}\label{eq:inter:S2}
      \begin{align}
        \norm{u-Iu}{L^\infty(K)}&\lesssim  (h+N^{-1}\max|\psi'|)^{k+1},\\
        \norm{\pt_x(u-Iu)}{L^\infty(K)}&\lesssim ((h+N^{-1})^k+\eps^{-1}(N^{-1}\max|\psi'|)^k).
      \end{align}
    \end{subequations}
    Combining the individual results finishes the proof.
  \end{proof}

  \subsubsection{The interpolation error on Dur\'{a}n meshes}

  \begin{lemma}
    For the standard piecewise Lagrange interpolation operator on the graded mesh \eqref{Dmesh1}, \eqref{Dmesh2} we have
    \begin{subequations}\label{eq:inter:D}
      \begin{align}
        \norm{u-Iu}{\beta_e}
          & \lesssim H^{k+1}
            \lesssim N^{-(k+1)}(\ln{(1/{\eps})})^{k+1},\label{eq:inter:D1}\\[0.5ex]
        \tnorm{u-Iu}_{\beta_e}
          & \lesssim (\sqrt{\eps}+H)H^k
            \lesssim (\sqrt{\eps}+N^{-1}\ln(1/\eps))N^{-k}(\ln(1/\eps))^k,\label{eq:inter:D2}\\[0.5ex]
        \norm{u-Iu}{\beta_b}
          & \lesssim H^{k+1}
            \lesssim N^{-(k+1)}(\ln{(1/{\eps})})^{k+1},\label{eq:inter:D3}\\[0.5ex]
        \tnorm{u-Iu}_{\beta_b}
          & \lesssim H^k
            \lesssim N^{-k}(\ln{(1/{\eps})})^k.\label{eq:inter:D4}
      \end{align}
    \end{subequations}
  \end{lemma}
  \begin{proof}
    In the proof we use the norm definitions, the solution decomposition \eqref{sd}
%     anisotropic estimates from \cite{Apel99}
    with bounds \eqref{s},\eqref{exp} and standard interpolation error estimates.
    Let us start with the unweighted norm on the interval $[0,1/2]$ and, using \eqref{eq:step}, obtain
    \begin{align*}
      \|s-Is\|_{L^2(0,1/2)}^2
        =&   \sum\limits_{i=1}^M\|s-Is\|_{L^2(I_i)}^2\\
        &\lesssim h_1^{2(k+1)}\|s^{(k+1)}\|_{L^2(I_1)}^2
          +\sum\limits_{i=2}^M H^{2(k+1)}\|x^{k+1}s^{(k+1)}\|_{L^2(I_i)}^2\\
        &\lesssim H^{2(k+1)}
    \end{align*}
    where $I_i=(x_{i-1},x_i).$ 
For the layer component $w_1$ we get 
\begin{align*}
      \|w_1-Iw_1\|_{L^2(0,1/2)}^2
        =&   \sum\limits_{i=1}^M\|w_1-Iw_1\|_{L^2(I_i)}^2\\
        &\lesssim h_1^{2(k+1)}\|w_1^{(k+1)}\|_{L^2(I_1)}^2
          +\sum\limits_{i=2}^M H^{2(k+1)}\|x^{k+1}w_1^{(k+1)}\|_{L^2(I_i)}^2\\
         & \lesssim (H\eps)^{2(k+1)}\int\limits_0^{H\eps}\eps^{-2(k+1)}e^{-\frac{2\alpha x}{\eps}}dx+H^{2(k+1)}\int\limits_{H\eps}^{1/2}x^{2(k+1)}\eps^{-2(k+1)}e^{-\frac{2\alpha x}{\eps}}dx\\
        &\lesssim \eps H^{2(k+1)}.
    \end{align*}   
    The same estimates we obtained for other layer components, $w_2$ and $w_3$,
 %   In the same way we get $\|w-Iw\|_{L^2(0,1/2)}^2\lesssim \eps H^{2(k+1)},$ for the layer components $w=w_1+w_2+w_3$ %$\|w_2-Iw_2\|_{L^2(0,1/2)}^2\lesssim \eps H^{2(k+1)},$ $\|w_3-Iw_3\|_{L^2(0,1/2)}^2\lesssim \eps H^{2(k+1)}$
    and similarly these results are valid on whole domain $[0,2]$. Thus, we conclude $\norm{u-Iu}{\beta_e}\lesssim H^{k+1}$.
    Analogously, the other estimates for the $H^1$-seminorm error $|s-Is|_{H^1(\Omega)}\lesssim H^k$ and $|w-Iw|_{H^1(\Omega)}\lesssim \eps^{-\frac{1}{2}}H^k,$ follow, needed for \eqref{eq:inter:D2}.
    So, we get $$\tnorm{u-Iu}_{\beta_e}\lesssim   (\sqrt{\eps}+H)H^k,$$ and with \eqref{H} we have \eqref{eq:inter:D2}.

    To get the estimates in the weighted $\beta_b$ norm, we start with $p\in(1,\infty)$ and
    \[
      \|u-Iu\|_{L^p(\Omega)}\leq  \|s-Is\|_{L^p(\Omega)}+ \|w-Iw\|_{L^p(\Omega)}.
    \]
    It follows like before $ \|s-Is\|_{L^p(\Omega)} \lesssim H^{k+1}$ and
    \begin{align*}
      \|w-Iw\|_{L^p(0,1/2)}^p
        &\lesssim H^{p(k+1)}\int\limits_{0}^{1/2}\Big(x^{k+1}\eps^{-(k+1)}e^{-\frac{\alpha x}{\eps}}\Big)^p dx\\
        &\lesssim \eps H^{p(k+1)}\int\limits_{0}^{1/(2\eps)}\Big(\hat x^{k+1}e^{-\alpha \hat x}\Big)^pd\hat x
         \lesssim \eps H^{p(k+1)},
    \end{align*}
    where a substitution $\hat x=x/\eps$ was used.
    With similar results in the remaining domain we arrive at
    \[
      \|u-Iu\|_{L^p(\Omega)}\lesssim H^{k+1}.
    \]
%   {\color{red} Is it OK that this part} Now, %from \cite[Theorem 2.8]{Adams} we obtain
 %   \begin{equation}\label{eq:1}
 %     \|u-Iu\|_{L^\infty(\Omega)}
 %       = \lim_{p\to\infty}\|u-Iu\|%_{L^p(\Omega)}\leq CH^{k+1}
%        \lesssim H^{k+1},
 %   \end{equation} where the constant $C$ is %independent of $p$.{\color{red} replace with\\
    For $p=\infty$ we have directly
    \[\|u-Iu\|_{L^\infty(I_i)}\leq C|I_i|^{k+1}|u|_{W^{k+1,\infty}(I_i)}\leq C\eps^{-(k+1)}H^{k+1}\|x^{k+1}e^{-\alpha  x/\eps}\|_{L^{\infty}(I_i)},\] so 
    \begin{align}\label{eq:1}
    \|u-Iu\|_{L^\infty(0,1/2)}&\leq  C\eps^{-(k+1)}H^{k+1}\|x^{k+1}e^{-\alpha  x/\eps}\|_{L^{\infty}(0,1/2)}\\ \nonumber
   & \leq  C\eps^{-(k+1)}H^{k+1}((k+1)\eps/\alpha)^{k+1}e^{-(k+1)}=CH^{k+1}\lesssim H^{k+1}.
    \end{align}

    Similarly, it is valid that $\|\partial_x (s-Is)\|_{L^p(\Omega)}\lesssim H^k$ and
    \begin{align*}
      \|\partial_x (w-Iw)\|_{L^p(0,1/2)}^p
        &\lesssim H^{pk}\int\limits_{0}^{1/2}\Big(x^k\eps^{-(k+1)}e^{-\frac{\alpha x}{\eps}}\Big)^p dx\\
        &\lesssim \eps H^{pk}\int\limits_{0}^{1/(2\eps)}\Big(\hat x^k\eps^{-1}e^{-\alpha\hat x}\Big)^pd\hat x
         \lesssim \eps^{1-p}H^{pk}
    \end{align*}
    and analogously in the remaining domain.
    For $p\rightarrow \infty$ we get
    \begin{equation}\label{eq:2}
      \|\partial_x(u-Iu)\|_{L^\infty(\Omega)}\lesssim \eps^{-1}H^{k}.
    \end{equation}
    Now, estimates \eqref{eq:inter:D3} and \eqref{eq:inter:D4} follow from \eqref{eq:1} and \eqref{eq:2} with \eqref{H}.
  \end{proof}

%   By combining \eqref{eq:conv} and the interpolation error results  the convergence result follows like
%   in the previous section.
%   \begin{lemma}\label{lem:main:D}
%     For the standard piecewise Lagrange interpolation operator $I$ on a Duran mesh with
%     we obtain
%     \begin{align}\label{eq:inter:TD2}
%       \tnorm{u-Iu}_{\beta_e}&\lesssim (\sqrt{\eps}+N^{-1}\ln(1/\eps))(N^{-1}\ln(1/\eps))^k,\\
%       \tnorm{u-Iu}_{\beta_b}&\lesssim (N^{-1}\ln(1/\eps))^k.
%     \end{align}
%   \end{lemma}

  \section{The full discretisation}\label{sec:full}
  Having the results for the stationary problem, we address now the time-dependent one.
  We have to look at the solution decomposition,
  the definition of the meshes and finally at the error analysis.
  Let us start with the solution decomposition.

  We assume basically the same structure as in the stationary case:
  \begin{subequations}\label{eq:decomp:time}
  \begin{equation}
    u=s+w_1+w_2+w_3,
  \end{equation}
  where for a given $k,q>0$ and all $\ell\in\{0,\dots,k+1\}$, $r\in\{0,\dots,q+1\}$ it holds
  \begin{align}
    \norm{\pt_t^r\pt_x^\ell s}{L^\infty(0,2)}&\lesssim 1,\\
    |\pt_t^r\pt_x^\ell w_1(x,t)|&\lesssim \eps^{-\ell}\e^{-\frac{\alpha x}{\eps}},&
    |\pt_t^r\pt_x^\ell w_2(x,t)|&\lesssim \eps^{-\ell}\e^{-\frac{\alpha |x-1|}{\eps}},&
    |\pt_t^r\pt_x^\ell w_3(x,t)|&\lesssim \eps^{-\ell}\e^{-\frac{\alpha (2-x)}{\eps}}.
  \end{align}
  \end{subequations}
  Thus, we assume the time-derivatives not to influence the upper bounds. This might not always be the case
  and depends on compatibility conditions on the data, see also \cite{CJLS98} for an analysis on reaction-diffusion problems.

  The temporal mesh is now given by dividing $[0,T]$ into $M$ cells $0=t_0<t_1<\dots<t_M=T$
  with individual width $\tau_m=t_m-t_{m-1}$ and maximal width $\tau:=\max\{\tau_m\}$. As spatial mesh we use the same ones
  described in the stationary case.

  We want to apply a discontinuous Galerkin method in time and a $\beta$-weighted continuous Galerkin FEM in space.
  Therefore, the discrete space is defined as
  \[
    \U_h^\tau:=\{U\in L^2(0,T;H^1_0(\Omega)):U|_{(t_{m-1},t_m)}\in\PS_q((t_{m-1},t_m);\U_h)\forall m\in\{1,\dots,M\}\}.
  \]
  Our method now reads: Find $U_h^\tau\in\U_h^\tau$ such that  for each $m\in\{1,\dots,M\}$ and $V\in\U_h^\tau$ it holds
  \[
    \Qm{\scp{\pt_t U_h^\tau,V}_{\beta,\Omega}}
    +\Qm{B_\beta(U_h^\tau,V)}
    +\scp{\jump{U_h^\tau}_{m-1},V^+_{m-1}}_{\beta,\Omega}
    = \Qm{F(V)},
  \]
  where

  \begin{equation}\label{Qm}
    \Qm{v} := \frac{\tau_m}{2} \sum_{i=0}^q \hat{\omega}_i v(\tmi)
  \end{equation}
  with the transformed Gau\ss--Radau points $\tmi := T_m(\hat{t}_i)$, $i\in\{0,\dots,q\}$
  and the Gauß-Radau weights $\hat{\omega}_i$,
  is a quadrature formula on $(t_{m-1},t_m]$ which is exact for polynomials of degree at most $2q$
  and
  \[
    v_m^{\pm} := \lim\limits_{t\to t_m\pm 0} v(t),\qquad
    \jump{v}_m := v_m^{+} - v_m^{-}
  \]
  are the one-sided limits and the jump at $t=t_m$, resp. Note that $U_0^{-}\in \U_h$ is a suitable
  approximation of the initial condition~$u_0$.

  Using this setup, we can immediately apply the results of \cite{FrM16} after checking the assumptions given therein.
  \begin{itemize}
    \item Assumption 1
          \[
            \norm{v}{\beta}\lesssim \tnorm{v}_\beta
          \]
          for all $v\in H^1_0(\Omega)$ is true due to %the conditions \eqref{eq:beta} on $\beta$ and
          the definition of the weighted triple norm.
    \item Assumption 2
          \begin{align*}
            \norm{v-Iv}{\beta} &\lesssim g_{L^2}(N),&
            \tnorm{v-Iv}_\beta &\lesssim g_e(N)
          \end{align*}
          is true for all $v\in H^1_0(\Omega)\cap H^{k+1}(\Omega)$ with
          \begin{equation}\label{gS}
            g_{L^2}(N)=(h+N^{-1}\max|\psi'|)^{k+1},\quad
            g_{e}(N)=(h+N^{-1}\max|\psi'|)^{k}
%             g_{e}(N)=\eps h^k+(N^{-1}\max|\psi'|)^{k}
          \end{equation}
          on an S-type mesh, see \eqref{eq:inter:S} -- \eqref{eq:inter:S2}, and
          \begin{equation}\label{gD}
            g_{L^2}(N)=N^{-(k+1)}(\ln(1/\eps))^{k+1},\quad
            g_{e}(N)=N^{-k}(\ln(1/\eps))^k
          \end{equation}
          on a Dur\'{a}n mesh, see \eqref{eq:inter:D}.
    \item Assumption 3
          \begin{align*}
            B_\beta(v,v) &\geq c\tnorm{v}_\beta^2,&
            B_\beta(\eta,v_N) &\lesssim g_d(N)\tnorm{v_N}_\beta
          \end{align*}
          holds true for all $v\in H^1_0(\Omega)$ and $v_N\in \U_h$ by \eqref{eq:coer} with $c=\frac{1}{2}$ and $g_d(N)=g_e(N)$, see \eqref{eq:bound}.
    \item Assumption 4 is true, because no stabilisation method is used.
    \item Assumption 5 deals with the time behaviour
          \[
            \norm{\pt_t^s\pt_x^k u}{L^2((0,T)\times\Omega)}\lesssim \norm{\pt_x^k u}{L^2((0,T)\times\Omega)}
          \]
          and holds true by our assumption on the solution decomposition \eqref{eq:decomp:time}.
  \end{itemize}
  Thus, \cite[Theorems 3.7, 3.9]{FrM16} give the main result of this paper.
  \begin{theorem}\label{thm:main}
    Under suitable regularity assumptions to enable the solution decomposition and using (\ref{Qm})-(\ref{gD}) we have on an S-type and Dur\'{a}n mesh
    in space 
    \[
      \sup_{t\in(0,T)}\norm{(u-U_h^\tau)(t)}{\beta}+
      \tnorm{u-U_h^\tau}_{Q,\beta}
      \lesssim T(\tau^{q+1}+ g_e(N)),%(h+N^{-1}\max|\psi'|)^{k}),
    \]
    where
    \[
      \tnorm{v}_{Q,\beta}^2
        := \sum_{m=1}^M \Qm{\tnorm{v(t)}_\beta^2}.
    \]
  \end{theorem}

  \begin{remark}
    For non-stationary problems and standard FEM in space it is an open problem to prove uniform estimates in a balanced norm.
    The weighted technique based on \cite{MS21} and presented here has the big advantage to yield such a balanced error estimation.
    This technique is not restricted to the spatial shift problem studied here, but
    can also be applied to other time-dependent singularly perturbed problems.
  \end{remark}

  \section{Numerical results}\label{sec:numerics}

  To illustrate our theoretical results, let us consider the following problem $Lu=f$ in $(0,2)\times(0,T)$ with non-constant coefficients
  \begin{align*}
    Lu(x,t) &= \partial_t u(x,t)-\varepsilon^2 \pt_{xx}u(x,t) - \left(1+\frac{1}{2}x^2\right)u(x-1,t) + 2\cosh(x-1)u(x,t),\\%
    f(x,t)  &= \mathrm{e}^{\frac{1}{2}x}% \quad \text{for } x\in[0,2],\tag{*}\\
%   \end{align*}
  \intertext{with initial and boundary conditions}
%   \begin{align*}
    u(x,0)  &= 0\mbox{ for }x\in [0,2], \\
    u(x,t)  &= \varphi(x,t)\mbox{ for }x\in[-1,0]\mbox{ and }t\in[0,T], \\
    u(1,t)  &= 0\mbox{ for }t\in[0,T],	
  \end{align*}
  with a function $\varphi$ specified below.
  An exact solution of this problem is not known. Therefore, we use a numerically computed reference solution as a substitute for the exact
  solution in computing our errors. For this, both in space and time, we use discrete FE-spaces on meshes with twice the number of cells
  and with polynomial degrees by one larger than the largest one considered.

  For numerical approximations using degrees $q$ and $k$ the theoretical findings propose our method to converge with order $q+1$
  in time and $k$ in space. Thus, equilibrating, we choose $k = q+1$ throughout our calculations. Moreover we use an
  equidistant mesh in time with $N=4M$ cells and therefore coupling spatial and temporal resolution.

  Let us begin by considering the homogeneous case, i.e. $\varphi\equiv 0$. For
%   two different diffusion parameter,   $\varepsilon = 0.04$ and
  $\varepsilon = 10^{-4}$, the resulting errors in the energy norm and the $L_2$-norm are
  displayed in Table~\ref{tab:1}. Here we additionally distinguish between the weighted case and
  the non-weighted case as described above.

  \begin{table}[htbp]
    \caption{First and second order of convergence for the homogeneous problem \label{tab:1}}
    \begin{center}

      \begin{tabular}{rllllllll}
          N & \multicolumn{2}{c}{$\norm{u-U_h^\tau}{\beta_e}$} & \multicolumn{2}{c}{$\tnorm{u-U_h^\tau}_{\beta_e}$}
            & \multicolumn{2}{c}{$\norm{u-U_h^\tau}{\beta_b}$} & \multicolumn{2}{c}{$\tnorm{u-U_h^\tau}_{\beta_b}$} \\
          \hline
          \multicolumn{9}{c}{$k = 1,\,q=0$}\\ %smooth f ??
          \hline
%			 8 & 1.35e-01 & 0.53 & 1.36e-01 & 0.53 & 3.90e-01 & 1.21 & 5.72e-01 & 0.34 \\
% 			16 & 9.35e-02 & 0.79 & 9.36e-02 & 0.79 & 1.69e-01 & 1.12 & 4.52e-01 & 0.81 \\
% 			32 & 5.39e-02 & 0.90 & 5.40e-02 & 0.90 & 7.78e-02 & 0.96 & 2.58e-01 & 0.96 \\
			64 & 2.89e-02 & 0.95 & 2.89e-02 & 0.95 & 4.00e-02 & 0.96 & 1.32e-01 & 0.99 \\
			128 & 1.49e-02 & 0.98 & 1.50e-02 & 0.98 & 2.06e-02 & 0.98 & 6.65e-02 & 1.00 \\
			256 & 7.59e-03 & 0.99 & 7.60e-03 & 0.99 & 1.05e-02 & 0.99 & 3.31e-02 & 1.01 \\
			512 & 3.83e-03 & 0.99 & 3.83e-03 & 0.99 & 5.27e-03 & 0.99 & 1.65e-02 & 1.01 \\
			1024 & 1.92e-03 &   & 1.92e-03 &   & 2.65e-03 &   & 8.20e-03 &   \\
          \hline
          \multicolumn{9}{c}{$k = 2,\,q=1$}\\ %smooth f ??
          \hline
%           8 & 2.97e-02 & 1.90 & 3.04e-02 & 1.85 & 2.56e-01 & 2.52 & 4.07e-01 & 1.08 \\
%           16 & 7.98e-03 & 1.95 & 8.43e-03 & 1.94 & 4.48e-02 & 2.61 & 1.92e-01 & 1.66 \\
%           32 & 2.06e-03 & 1.98 & 2.20e-03 & 1.98 & 7.35e-03 & 2.49 & 6.08e-02 & 1.90 \\
          64 & 5.22e-04 & 1.99 & 5.60e-04 & 1.99 & 1.31e-03 & 2.12 & 1.62e-02 & 1.95 \\
          128 & 1.31e-04 & 1.99 & 1.41e-04 & 1.99 & 3.01e-04 & 1.88 & 4.18e-03 & 1.90 \\
          256 & 3.30e-05 & 2.00 & 3.57e-05 & 1.97 & 8.18e-05 & 1.80 & 1.12e-03 & 1.74 \\
          512 & 8.26e-06 & 2.00 & 9.13e-06 & 1.90 & 2.35e-05 & 1.76 & 3.37e-04 & 1.49 \\
          1024 & 2.07e-06 &    & 2.45e-06 &    & 6.95e-06 &    & 1.20e-04 &    \\
      \end{tabular}
    \end{center}
  \end{table}

  Similar results are also obtained for different values of $\eps$ (see Table \ref{tab:3}), thereby confirming the robustness in $\eps$.

  \begin{table}[htbp]
  	\caption{$\varepsilon$-robustness for $k = 2$, $q = 1$, $\varphi\equiv 0$ and $\beta=\beta_e$ \label{tab:3}}
  	\begin{center}

  		\begin{tabular}{rllllllllll}
  			$N / \varepsilon$ & \multicolumn{2}{c}{$10^{-2}$} & \multicolumn{2}{c}{$10^{-4}$}
  			& \multicolumn{2}{c}{$10^{-6}$} & \multicolumn{2}{c}{$10^{-8}$} & \multicolumn{2}{c}{$10^{-10}$}\\
  			\hline
  			\multicolumn{9}{c}{$\norm{u-U_h^\tau}{\beta_e}$}\\ %smooth f ??
  			\hline
%  			8 & 4.644e-02 &   2.17& 2.968e-02 &   1.90& 2.946e-02 &   1.89& 2.946e-02 &   1.89& 2.946e-02 &   1.89\\
%   			16 & 1.029e-02 &   2.19& 7.977e-03 &   1.95& 7.950e-03 &   1.95& 7.950e-03 &   1.95& 7.950e-03 &   1.95\\
%   			32 & 2.257e-03 &   2.06& 2.058e-03 &   1.98& 2.055e-03 &   1.98& 2.055e-03 &   1.98& 2.055e-03 &   1.98\\
  			  64 & 5.40e-04 &   2.01& 5.22e-04 &   1.99& 5.22e-04 &   1.99& 5.22e-04 &   1.99& 5.22e-04 &   1.99\\
  			 128 & 1.34e-04 &   1.99& 1.32e-04 &   1.99& 1.31e-04 &   1.99& 1.31e-04 &   1.99& 1.31e-04 &   1.99\\
  			 256 & 3.38e-05 &   1.98& 3.30e-05 &   2.00& 3.30e-05 &   2.00& 3.30e-05 &   2.00& 3.30e-05 &   2.00\\
  			 512 & 8.55e-06 &   1.98& 8.26e-06 &   2.00& 8.26e-06 &   2.00& 8.26e-06 &   2.00& 8.26e-06 &   2.00\\
  			1024 & 2.17e-06 &       & 2.07e-06 &       & 2.07e-06 &       & 2.07e-06 &       & 2.07e-06 &    \\
  			\hline
  			\multicolumn{9}{c}{$\tnorm{u-U_h^\tau}_{\beta_e}$}\\ %smooth f ??
  			\hline
%  			8 & 7.873e-02 &   1.43& 3.036e-02 &   1.85& 2.947e-02 &   1.89& 2.946e-02 &   1.89& 2.946e-02 &   1.89\\
%   			16 & 2.915e-02 &   1.83& 8.431e-03 &   1.94& 7.955e-03 &   1.95& 7.950e-03 &   1.95& 7.950e-03 &   1.95\\
%   			32 & 8.193e-03 &   1.96& 2.203e-03 &   1.98& 2.057e-03 &   1.98& 2.055e-03 &   1.98& 2.055e-03 &   1.98\\
  			  64 & 2.11e-03 &   1.98& 5.60e-04 &   1.99& 5.22e-04 &   1.99& 5.22e-04 &   1.99& 5.22e-04 &   1.99\\
  			 128 & 5.35e-04 &   1.93& 1.41e-04 &   1.99& 1.32e-04 &   1.99& 1.31e-04 &   1.99& 1.31e-04 &   1.99\\
  			 256 & 1.40e-04 &   1.81& 3.57e-05 &   1.97& 3.30e-05 &   2.00& 3.30e-05 &   2.00& 3.30e-05 &   2.00\\
  			 512 & 3.99e-05 &   1.59& 9.13e-06 &   1.90& 8.27e-06 &   2.00& 8.26e-06 &   2.00& 8.26e-06 &   2.00\\
  			1024 & 1.32e-05 &       & 2.45e-06 &       & 2.07e-06 &       & 2.07e-06 &       & 2.07e-06 &    \\
  		\end{tabular}
  	\end{center}
  \end{table}

  For degree one and two we observe in those two tables the theoretical uniform orders w.r.t. $\eps$.
  Interestingly, as can be observed in Figure \ref{fig:HOC}
  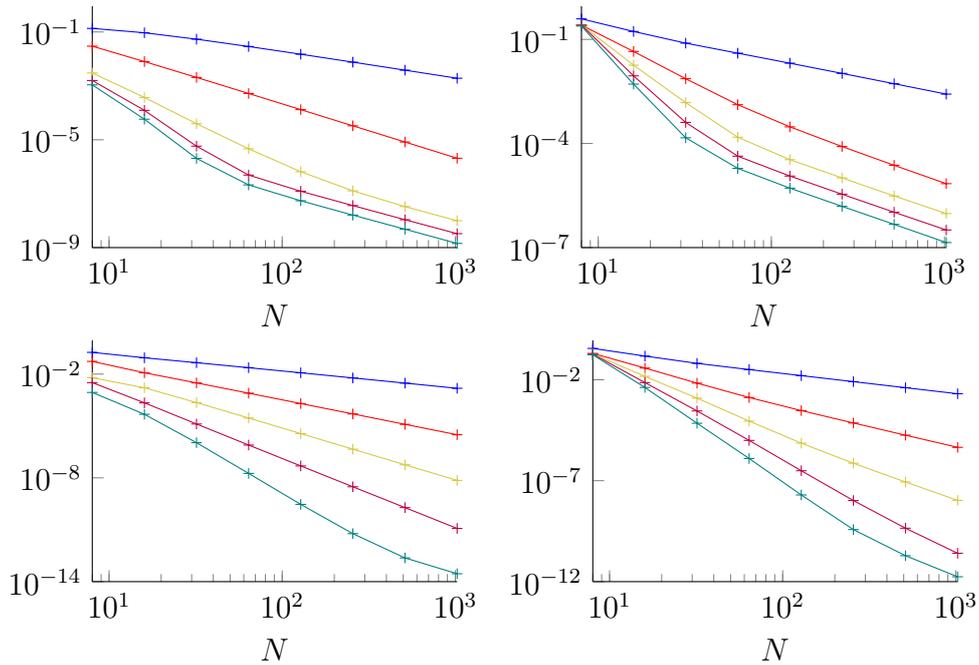
\begin{figure}[htbp]
    \begin{center}
      \input{non_smooth_non_weighted}
      \input{non_smooth_weighted}\\
      \input{smooth_non_weighted}
      \input{smooth_weighted}
    \end{center}
%   	\begin{center}
%   		\includegraphics[scale=0.5]{pics/s0w0p0}
%   		\includegraphics[scale=0.5]{pics/s0w1p0}\\
%   		\includegraphics[scale=0.5]{pics/s1w0p0}
%   		\includegraphics[scale=0.5]{pics/s1w1p0}
%   	\end{center}
  	\caption{Convergence results in $\tnorm{\cdot}_{\beta}$ for $\eps=10^{-4}$, $\phi\equiv0$, $k=1$ (blue) to $k=5$ (teal), and
  	$\beta=\beta_e$ (left) and $\beta=\beta_b$ (right), and
  	right-hand sides $f(x,t)=\mathrm{e}^{\frac{1}{2}x}$ (upper) and $f(x,t)=\mathrm{e}^{\frac{1}{2}x}\cdot\begin{cases}t^2(12-16t),&t<1/2\\1,&t\geq1/2\end{cases}$(lower).
  	\label{fig:HOC}}
  \end{figure}
  for higher polynomial degrees, we observe a decrease
  in the convergence orders, if the right-hand side is not decreasing fast enough to zero in the corners of our
  space-time domain. This seems to imply a loss of regularity in the solution, and thereby our assumption on the
  solution decomposition is no longer true. For smaller values of $\eps$ this reduction is less dominant.
  
  Furthermore,
  the balanced method emphasises this stronger than the standard weigh\-ted method, as can be seen in the figure
  and also in Table~\ref{tab:1} for $k=2$.
  With an inhomogeneous shift data like $\varphi(x) = 3x^2$ one also attains the theoretical orders of convergence w.r.t. $\eps$,
  see Table~\ref{tab:2}. %where for higher orders similarly are compromised by corner effects that cause the numerical
  %solution to not satisfy our assumptions.
  %Moreover, our data suggests that the smaller the chosen $\varepsilon$ the less pronounced this effect seems to be.
  \begin{table}[htbp]
	\caption{First and second order of convergence for $\varphi(x) = 3x^2$ \label{tab:2}}
    \begin{center}

      \begin{tabular}{rllllllll}
          N & \multicolumn{2}{c}{$\norm{u-U_h^\tau}{\beta_e}$} & \multicolumn{2}{c}{$\tnorm{u-U_h^\tau}_{\beta_e}$}
            & \multicolumn{2}{c}{$\norm{u-U_h^\tau}{\beta_b}$} & \multicolumn{2}{c}{$\tnorm{u-U_h^\tau}_{\beta_b}$} \\
          \hline
          \multicolumn{9}{c}{$k = 1,\,q=0$}\\ %smooth f ??
          \hline
%			   8 & 1.85e-01 & 0.57 & 1.85e-01 & 0.57 & 5.39e-01 & 1.26 & 7.99e-01 & 0.37 \\
% 			16 & 1.24e-01 & 0.80 & 1.25e-01 & 0.80 & 2.24e-01 & 1.13 & 6.18e-01 & 0.81 \\
% 			32 & 7.15e-02 & 0.90 & 7.16e-02 & 0.90 & 1.03e-01 & 0.96 & 3.53e-01 & 0.96 \\
			64 & 3.83e-02 & 0.95 & 3.84e-02 & 0.95 & 5.29e-02 & 0.96 & 1.82e-01 & 0.99 \\
			128 & 1.98e-02 & 0.98 & 1.98e-02 & 0.98 & 2.73e-02 & 0.97 & 9.14e-02 & 1.00 \\
			256 & 1.01e-02 & 0.99 & 1.01e-02 & 0.99 & 1.39e-02 & 0.99 & 4.56e-02 & 1.01 \\
			512 & 5.08e-03 & 0.99 & 5.09e-03 & 0.99 & 7.02e-03 & 0.99 & 2.27e-02 & 1.01 \\
			1024 & 2.55e-03 &    & 2.56e-03 &    & 3.53e-03 &    & 1.13e-02 &    \\
          \hline
          \multicolumn{9}{c}{$k = 2,\,q=1$}\\ %smooth f ??
          \hline
%              8 & 3.97e-02 & 1.84 & 4.07e-02 & 1.80 & 3.59e-01 & 2.54 & 5.69e-01 & 1.10 \\
%             16 & 1.11e-02 & 1.93 & 1.17e-02 & 1.91 & 6.20e-02 & 2.56 & 2.66e-01 & 1.65 \\
%             32 & 2.91e-03 & 1.97 & 3.11e-03 & 1.96 & 1.05e-02 & 2.39 & 8.49e-02 & 1.89 \\
            64 & 7.45e-04 & 1.98 & 7.98e-04 & 1.98 & 2.01e-03 & 2.04 & 2.29e-02 & 1.93 \\
            128 & 1.88e-04 & 1.99 & 2.02e-04 & 1.98 & 4.88e-04 & 1.85 & 6.01e-03 & 1.85 \\
            256 & 4.74e-05 & 2.00 & 5.13e-05 & 1.95 & 1.36e-04 & 1.79 & 1.66e-03 & 1.67 \\
            512 & 1.19e-05 & 2.00 & 1.33e-05 & 1.87 & 3.92e-05 & 1.76 & 5.23e-04 & 1.43 \\
            1024 & 2.97e-06 &    & 3.63e-06 &    & 1.16e-05 &    & 1.94e-04 &    \\
      \end{tabular}
    \end{center}
  \end{table}

%   Finally, to get an impression of the structure of the solution with inhomogeneous boundary condition the reader is invited to have look at Figure \ref{fig:1}. Here, for two values of $\varepsilon$, we depict both the spatial solution at final time $t = T=1$, as well as a representation of the entire solution in space-time.
%
%   \begin{figure}[ht]
%     \begin{center}
%       \includegraphics[scale=0.4]{pics/solPhi0_04_.png}
%       \includegraphics[scale=0.4]{pics/solPhi1em3_.png}\\
%       \includegraphics[scale=0.4]{pics/solPhi0_04.png}
%       \includegraphics[scale=0.4]{pics/solPhi1em3.png}
%     \end{center}
%     \caption{Approximate solution %in $V_{1/80, 3}^{1/192, 2}$
%             for $\varepsilon=0.04 \text{ and } \varepsilon=0.001$ \label{fig:1}}
%   \end{figure}

  \vspace{1em}
  {\bf Acknowledgment}. This paper was written during a visit by Mirjana Brdar to the Technical University of Dresden in August -- October
  2021 supported by DAAD grant, number 57552334 and partly supported by Ministry of Education, Science and Technological Development of the Republic of Serbia under
  grant no 451-03-9/2021-14/200134 and bilateral project "Singularly perturbed problems with multiple parameters" between Germany and Serbia.

\end{document}

%% file: non_smooth_non_weighted.tex
% \definecolor{blue}{rgb}{0.00000,0.44700,0.74100}%
% \definecolor{red}{rgb}{0.85000,0.32500,0.09800}%
% \definecolor{yellow!80!black}{rgb}{0.92900,0.69400,0.12500}%
% \definecolor{purple}{rgb}{0.49400,0.18400,0.55600}%
% \definecolor{teal}{rgb}{0.46600,0.67400,0.18800}%
%
\begin{tikzpicture}

\begin{axis}[%
width=0.3\textwidth,
height=0.2\textwidth,
scale only axis,
xmode=log,xmin=8,xmax=1024,xlabel={$N$},
ymode=log,ymin=1e-09,ymax=0.9,
xminorticks=true,yminorticks=true,
axis background/.style={fill=white},
title style={font=\bfseries, align=center},
axis x line*=bottom,
axis y line*=left
]
\addplot [color=blue, mark=+, mark options={solid, blue}, forget plot]
  table[row sep=crcr]{%
8	0.135315208480472\\
16	0.093459060223656\\
32	0.0539358095509012\\
64	0.0288817525112369\\
128	0.0149346864935304\\
256	0.00759250337072408\\
512	0.00382742642735684\\
1024	0.00192115559581216\\
};
\addplot [color=red, mark=+, mark options={solid, red}, forget plot]
  table[row sep=crcr]{%
8	0.0296798338394518\\
16	0.00797711200662552\\
32	0.00205750269541077\\
64	0.000522037676293642\\
128	0.000131459773861016\\
256	3.29847244397533e-05\\
512	8.26183981728057e-06\\
1024	2.06775980878992e-06\\
};
\addplot [color=yellow!80!black, mark=+, mark options={solid, yellow!80!black}, forget plot]
  table[row sep=crcr]{%
8	0.00301191849922454\\
16	0.000372021173479375\\
32	3.94821938598665e-05\\
64	4.63053883502199e-06\\
128	6.53124626688498e-07\\
256	1.2737884020581e-07\\
512	3.32288958388529e-08\\
1024	9.9324816584793e-09\\
};
\addplot [color=purple, mark=+, mark options={solid, purple}, forget plot]
  table[row sep=crcr]{%
8	0.00158132715037468\\
16	0.000123630340843447\\
32	5.79426742530629e-06\\
64	4.9006275131498e-07\\
128	1.23362339080461e-07\\
256	3.62586945677627e-08\\
512	1.08262950241877e-08\\
1024	3.29179906918482e-09\\
};
\addplot [color=teal, mark=+, mark options={solid, teal}, forget plot]
  table[row sep=crcr]{%
8	0.00109092512569833\\
16	5.80203404411259e-05\\
32	2.03952524580991e-06\\
64	2.11626055583503e-07\\
128	5.46256584032861e-08\\
256	1.60509581285341e-08\\
512	4.80980882739079e-09\\
1024	1.42714699747054e-09\\
};
\end{axis}
\end{tikzpicture}%

%% file: non_smooth_weighted.tex
% \definecolor{blue}{rgb}{0.00000,0.44700,0.74100}%
% \definecolor{red}{rgb}{0.85000,0.32500,0.09800}%
% \definecolor{yellow!80!black}{rgb}{0.92900,0.69400,0.12500}%
% \definecolor{purple}{rgb}{0.49400,0.18400,0.55600}%
% \definecolor{teal}{rgb}{0.46600,0.67400,0.18800}%
%
\begin{tikzpicture}

\begin{axis}[%
width=0.3\textwidth,
height=0.2\textwidth,
scale only axis,
xmode=log,xmin=8,xmax=1024,xlabel={$N$},
ymode=log,ymin=1e-07,ymax=0.9,
xminorticks=true,yminorticks=true,
axis background/.style={fill=white},
title style={font=\bfseries, align=center},
axis x line*=bottom,
axis y line*=left
]
\addplot [color=blue, mark=+, mark options={solid, blue}, forget plot]
  table[row sep=crcr]{%
8	0.389536638067543\\
16	0.168674067481597\\
32	0.0778152899651142\\
64	0.03999954229656\\
128	0.0205578236450838\\
256	0.0104530938649471\\
512	0.00527383258857542\\
1024	0.00264859958269374\\
};
\addplot [color=red, mark=+, mark options={solid, red}, forget plot]
  table[row sep=crcr]{%
8	0.25644066725532\\
16	0.0447727040342133\\
32	0.00734704604903332\\
64	0.00130993461407785\\
128	0.000301266046630698\\
256	8.17996385433786e-05\\
512	2.35083372247577e-05\\
1024	6.94618282709945e-06\\
};
\addplot [color=yellow!80!black, mark=+, mark options={solid, yellow!80!black}, forget plot]
  table[row sep=crcr]{%
8	0.261296154634287\\
16	0.0180308259165261\\
32	0.00150660814606859\\
64	0.000151048153022991\\
128	3.4204555276096e-05\\
256	1.01659319854436e-05\\
512	3.07277699538923e-06\\
1024	9.57803581502343e-07\\
};
\addplot [color=purple, mark=+, mark options={solid, purple}, forget plot]
  table[row sep=crcr]{%
8	0.258683596505945\\
16	0.00889327070161178\\
32	0.000403753398158284\\
64	4.30576694945731e-05\\
128	1.15221086715029e-05\\
256	3.4548199859533e-06\\
512	1.04641390667156e-06\\
1024	3.21617142803307e-07\\
};
\addplot [color=teal, mark=+, mark options={solid, teal}, forget plot]
  table[row sep=crcr]{%
8	0.241680954969328\\
16	0.00513211799025273\\
32	0.000145968188375703\\
64	1.91698751236003e-05\\
128	5.1639747985925e-06\\
256	1.54263327854629e-06\\
512	4.67845759742015e-07\\
1024	1.3993750122705e-07\\
};
\end{axis}
\end{tikzpicture}%

%% file: smooth_non_weighted.tex
% \definecolor{blue}{rgb}{0.00000,0.44700,0.74100}%
% \definecolor{red}{rgb}{0.85000,0.32500,0.09800}%
% \definecolor{yellow!80!black}{rgb}{0.92900,0.69400,0.12500}%
% \definecolor{purple}{rgb}{0.49400,0.18400,0.55600}%
% \definecolor{teal}{rgb}{0.46600,0.67400,0.18800}%
%
\begin{tikzpicture}

\begin{axis}[%
width=0.3\textwidth,
height=0.2\textwidth,
scale only axis,
xmode=log,xmin=8,xmax=1024,xlabel={$N$},
ymode=log,ymin=1e-14,ymax=0.9,
xminorticks=true,yminorticks=true,
axis background/.style={fill=white},
title style={font=\bfseries, align=center},
axis x line*=bottom,
axis y line*=left
]
\addplot [color=blue, mark=+, mark options={solid, blue}, forget plot]
  table[row sep=crcr]{%
8	0.180222719485403\\
16	0.0876815777279349\\
32	0.045515123627484\\
64	0.0234131843268256\\
128	0.0118981127014027\\
256	0.0060005066586239\\
512	0.00301381969424983\\
1024	0.00151065430828045\\
};
\addplot [color=red, mark=+, mark options={solid, red}, forget plot]
  table[row sep=crcr]{%
8	0.0541862931522343\\
16	0.0119248074015079\\
32	0.00309742441071584\\
64	0.000789022681202713\\
128	0.000198690438768714\\
256	4.98201761775478e-05\\
512	1.24713646672475e-05\\
1024	3.1197426176481e-06\\
};
\addplot [color=yellow!80!black, mark=+, mark options={solid, yellow!80!black}, forget plot]
  table[row sep=crcr]{%
8	0.00612896197455106\\
16	0.00162068069617364\\
32	0.000224751742372772\\
64	2.88480472963885e-05\\
128	3.63497525691669e-06\\
256	4.55622449421857e-07\\
512	5.70133247003888e-08\\
1024	7.12991383101868e-09\\
};
\addplot [color=purple, mark=+, mark options={solid, purple}, forget plot]
  table[row sep=crcr]{%
8	0.00320401058559875\\
16	0.000218746688035149\\
32	1.30439431941783e-05\\
64	7.89791591886442e-07\\
128	4.89692512217828e-08\\
256	3.05727136280723e-09\\
512	1.91117703600114e-10\\
1024	1.19482109162023e-11\\
};
\addplot [color=teal, mark=+, mark options={solid, teal}, forget plot]
  table[row sep=crcr]{%
8	0.000870816962923113\\
16	4.74959441478922e-05\\
32	1.08309195681439e-06\\
64	1.81005798710303e-08\\
128	2.92835962291042e-10\\
256	5.95251882431158e-12\\
512	2.35270698870247e-13\\
1024	2.87340225635346e-14\\
};
\end{axis}
\end{tikzpicture}%

%% file: smooth_weighted.tex
% \definecolor{blue}{rgb}{0.00000,0.44700,0.74100}%
% \definecolor{red}{rgb}{0.85000,0.32500,0.09800}%
% \definecolor{yellow!80!black}{rgb}{0.92900,0.69400,0.12500}%
% \definecolor{purple}{rgb}{0.49400,0.18400,0.55600}%
% \definecolor{teal}{rgb}{0.46600,0.67400,0.18800}%
%
\begin{tikzpicture}

\begin{axis}[%
width=0.3\textwidth,
height=0.2\textwidth,
scale only axis,
xmode=log,xmin=8,xmax=1024,xlabel={$N$},
ymode=log,ymin=1e-12,ymax=0.9,
xminorticks=true,yminorticks=true,
axis background/.style={fill=white},
title style={font=\bfseries, align=center},
axis x line*=bottom,
axis y line*=left
]
\addplot [color=blue, mark=+, mark options={solid, blue}, forget plot]
  table[row sep=crcr]{%
8	0.356250423027417\\
16	0.149044059834244\\
32	0.0648435606493117\\
64	0.0318677882975527\\
128	0.0160035532352378\\
256	0.00804569427739693\\
512	0.00403741701471074\\
1024	0.0020231746581052\\
};
\addplot [color=red, mark=+, mark options={solid, red}, forget plot]
  table[row sep=crcr]{%
8	0.201681402065384\\
16	0.0379445858625874\\
32	0.0067962537502321\\
64	0.00131839975448104\\
128	0.000297192296329283\\
256	7.22757659699196e-05\\
512	1.79616984535645e-05\\
1024	4.48591638882052e-06\\
};
\addplot [color=yellow!80!black, mark=+, mark options={solid, yellow!80!black}, forget plot]
  table[row sep=crcr]{%
8	0.188692481475902\\
16	0.0147492232446108\\
32	0.00122714870324117\\
64	8.98998181780906e-05\\
128	7.25033509167251e-06\\
256	7.34833122405184e-07\\
512	8.61877997862004e-08\\
1024	1.06081969280954e-08\\
};
\addplot [color=purple, mark=+, mark options={solid, purple}, forget plot]
  table[row sep=crcr]{%
8	0.185429211426633\\
16	0.00723583497444059\\
32	0.000285113096365774\\
64	9.88392363249846e-06\\
128	3.11606494460311e-07\\
256	1.04145130847412e-08\\
512	4.43352646299773e-10\\
1024	2.54063138335057e-11\\
};
\addplot [color=teal, mark=+, mark options={solid, teal}, forget plot]
  table[row sep=crcr]{%
8	0.172615314352457\\
16	0.00414793246536534\\
32	7.01974260845749e-05\\
64	1.25667376029437e-06\\
128	1.96656952043163e-08\\
256	3.79400092184954e-10\\
512	1.92718916798628e-11\\
1024	1.74507065522528e-12\\
};
\end{axis}
\end{tikzpicture}%